\documentclass{amsart}
\usepackage{amsfonts}
\begin{document}
\newtheorem{defn}{Definition}[section]
\newtheorem{thm}{Theorem}[section]
\newtheorem{prop}{Proposition}[section]
\newtheorem{exam}{Example}[section]
\newtheorem{cor}{Corollary}[section]
\newtheorem{rem}{Remark}[section]
\newtheorem{lem}{Lemma}[section]
\newcommand{\CC}{\mathbb{C}}
\newcommand{\KK}{\mathbb{K}}
\newcommand{\ZZ}{\mathbb{Z}}
\def\a{{\alpha}}
\def\b{{\beta}}
\def\d{{\delta}}
\def\g{{\gamma}}
\def\l{{\lambda}}
\def\gg{{\mathfrak g}}
\def\cal{\mathcal }

\title{The classification of Leibniz superalgebras of nilindex $n+m$ ($m\neq0.$)}
\author{J. R. G\'{o}mez, A.Kh. Khudoyberdiyev and B.A. Omirov}
\address{[J.R. G\'{o}mez] Dpto. Matem\'{a}tica Aplicada I.
Universidad de Sevilla. Avda. Reina Mercedes, s/n. 41012 Sevilla.
(Spain)} \email{jrgomez@us.es}
\address{[A.Kh. Khudoyberdiyev -- B.A. Omirov] Institute of Mathematics and
Information Technologies of Academy of Uzbekistan, 29, F.Hodjaev
srt., 100125, Tashkent (Uzbekistan)} \email{khabror@mail.ru ---
omirovb@mail.ru}

\thanks{The first author was supported by the PAI, FQM143 of the Junta de Andaluc\'{\i}a
(Spain) and the last author was supported by grant
NATO-Reintegration ref. CBP.EAP.RIG.983169}

\maketitle

\begin{abstract}
In this paper we investigate the description of the complex
Leibniz superalgebras with nilindex $n+m$, where $n$ and $m$
($m\neq 0$) are dimensions of even and odd parts, respectively. In
fact, such superalgebras with characteristic sequence equal to
$(n_1, \dots, n_k | m_1, \dots, m_s)$ (where $n_1+\dots +n_k=n, \
m_1+ \dots + m_s=m$) for $n_1\geq n-1$ and $(n_1, \dots, n_k | m)$
were classified in works \cite{FilSup}--\cite{C-G-O-Kh1}. Here we
prove that in the case of $(n_1, \dots, n_k| m_1, \dots, m_s)$,
where $n_1\leq n-2$ and $m_1 \leq m-1$ the Leibniz superalgebras
have nilindex less than $n+m.$ Thus, we complete the
classification of Leibniz superalgebras with nilindex $n+m.$
\end{abstract}



\section{Introduction}

During many years the theory of Lie superalgebras has been
actively studied by many mathematicians and physicists. A
systematic exposition of basic of Lie superalgebras theory can be
found in \cite{Kac}. Many works have been devoted to the study of
this topic, but unfortunately most of them do not deal with
nilpotent Lie superalgebras. In works \cite{2007Yu}, \cite{GL},
\cite{G-K-N} the problem of the description of some classes of
nilpotent Lie superalgebras have been studied. It is well known
that Lie superalgebras are a generalization of Lie algebras. In
the same way, the notion of Leibniz algebras, which were
introduced in \cite{Lod}, can be generalized to Leibniz
superalgebras \cite{Alb}, \cite{Liv}. Some elementary properties
of Leibniz superalgebras were obtained in \cite{Alb}.

In the work \cite{G-K-N} the Lie superalgebras with maximal
nilindex were classified. Such superalgebras are two-generated and
its nilindex equal to $n+m$ (where $n$ and $m$ are dimensions of
even and odd parts, respectively). In fact, there exists unique
Lie superalgebra of maximal nilindex. This superalgebra is
filiform Lie superalgebra (the characteristic sequence equal to
$(n-1,1 | m)$) and we mention about paper \cite{2007Yu}, where
some crucial properties of filiform Lie superalgebras are given.

For nilpotent Leibniz superalgebras the description of the case of
maximal nilindex (nilpotent Leibniz superalgebras distinguished by
the feature of being single-generated) is not difficult and was
done in \cite{Alb}.

However, the description of Leibniz superalgebras of nilindex
$n+m$ is a very problematic one and it needs to solve many
technical tasks. Therefore, they can be studied by applying
restrictions on their characteristic sequences. In the present
paper we consider Leibniz superalgebras with characteristic
sequence $(n_1, \dots, n_k | m_1, \dots, m_s)$ ($n_1\leq n-2$ and
$m_1\leq m-1$) and nilindex $n+m.$ Recall, that such superalgebras
for $n_1\geq n-1$ or $m_1=m$ have been already classified in works
\cite{FilSup}--\cite{C-G-O-Kh1}. Namely, we prove that a Leibniz
superalgebra with characteristic sequence equal to $(n_1, \dots,
n_k | m_1, \dots, m_s)$ ($n_1\leq n-2$ and $m_1\leq m-1$) has
nilindex less than $n+m.$ Therefore, we complete classification of
Leibniz superalgebras with nilindex $n+m.$

It should be noted that in our study the natural gradation of even
part of Leibniz superalgebra played one of the crucial roles. In
fact, we used some properties of naturally graded Lie and Leibniz
algebras for obtaining the convenience basis of even part of the
superalgebra (so-called adapted basis).

Throughout this work we shall consider spaces and (super)algebras
over the field of complex numbers. By asterisks $(*)$ we denote
the appropriate coefficients at the basic elements of
superalgebra.

\section{Preliminaries}

Recall the notion of Leibniz superalgebras.

\begin{defn} A $\mathbb{Z}_2$-graded vector space $L=L_0\oplus
L_1$ is called a Leibniz superalgebra if it is equipped with a
product $[-, -]$ which satisfies the following conditions:

1. $[L_\alpha,L_\beta]\subseteq L_{\alpha+\beta(mod\ 2)},$

2. $[x, [y, z]]=[[x, y], z] - (-1)^{\alpha\beta} [[x, z], y]-$
  Leibniz superidentity,\\
 for all $x\in L,$ $y \in L_\alpha,$ $z \in
L_\beta$ and $\alpha,\beta\in \mathbb{Z}_2.$
\end{defn}

The vector spaces $L_0$ and $L_1$ are said to be even and odd
parts of the superalgebra $L$, respectively. Evidently, even part
of the Leibniz superalgebra is a Leibniz algebra.

Note that if in Leibniz superalgebra $L$ the identity
$$[x,y]=-(-1)^{\alpha\beta} [y,x]$$ holds for any $x \in
L_{\alpha}$ and $y \in L_{\beta},$ then the Leibniz superidentity
can be transformed into the Jacobi superidentity. Thus, Leibniz
superalgebras are a generalization of Lie superalgebras and
Leibniz algebras.

The set of all Leibniz superalgebras with the dimensions of the
even and odd parts, respectively equal to $n$ and $m$, we denote
by $Leib_{n,m}.$

For a given Leibniz superalgebra $L$ we define the descending
central sequence as follows:
$$
L^1=L,\quad L^{k+1}=[L^k,L], \quad k \geq 1.
$$

\begin{defn} A Leibniz superalgebra $L$ is called
nilpotent, if there exists  $s\in\mathbb N$ such that $L^s=0.$ The
minimal number $s$ with this property is called nilindex of the
superalgebra $L.$
\end{defn}

\begin{defn} The set $$\mathcal{R}(L)=\left\{ z\in L\ |\ [L,
z]=0\right\}$$ is called the right annihilator  of a superalgebra
$L.$
\end{defn}

Using the Leibniz superidentity it is easy to see that
$\mathcal{R}(L)$ is an ideal of the superalgebra $L$. Moreover,
the elements of the form $[a,b]+(-1)^{\alpha \beta}[b,a],$ ($a \in
L_{\alpha}, \ b \in L_{\beta}$) belong to $\mathcal{R}(L)$.

The following theorem describes nilpotent Leibniz superalgebras
with maximal nilindex.
\begin{thm} \label{t1} \cite{Alb} Let $L$  be a Leibniz superalgebra
of $Leib_{n,m }$ with nilindex equal to $n+m+1.$ Then $L$ is
isomorphic to one of the following non-isomorphic superalgebras:
$$
[e_i,e_1]=e_{i+1},\ 1\le i\le n-1, \ m=0;\quad \left\{ \begin{array}{ll} [e_i,e_1]=e_{i+1},& 1\le i\le n+m-1, \\
{[}e_i,e_2{]}=2e_{i+2}, & 1\le i\le n+m-2,\\ \end{array}\right.
$$
(omitted products are equal to zero).
\end{thm}

\begin{rem} {\em
From the assertion of Theorem \ref{t1} we have that in case of
non-trivial odd part $L_1$ of the superalgebra $L$ there are two
possibility for $n$ and $m$, namely, $m=n$ if $n+m$ is even and
$m=n+1$ if $n+m$ is odd. Moreover, it is clear that the Leibniz
superalgebra has the maximal nilindex if and only if it is
single-generated.}
\end{rem}

Let $L=L_0\oplus L_1$ be a nilpotent Leibniz superalgebra. For an
arbitrary element $x\in L_0,$ the operator of right multiplication
$R_x:L \rightarrow L$ (defined as $R_x(y)=[y,x]$) is a nilpotent
endomorphism of the space $L_i,$ where $i\in \{0, 1\}.$ Taking
into account the property of complex endomorphisms we can consider
the Jordan form for $R_x.$ For operator $R_x$ denote by $C_i(x)$
($i\in \{0, 1\}$) the descending sequence of its Jordan blocks
dimensions. Consider the lexicographical order on the set
$C_i(L_0)$.

\begin{defn} \label{d4}A sequence
$$C(L)=\left( \left.\max\limits_{x\in L_0\setminus L_0^2} C_0(x)\
\right|\ \max\limits_{\widetilde x\in L_0\setminus L_0^2}
C_1\left(\widetilde x\right) \right) $$ is said to be the
characteristic sequence of the Leibniz superalgebra $L.$
\end{defn}

Similarly to \cite{GL} (corollary 3.0.1) it can be proved that the
characteristic sequence is invariant under isomorphism.

Since Leibniz superalgebras from $Leib_{n,m}$ with nilindex $n+m$
and with characteristic sequences equal to $(n_1, \dots, n_k |
m_1, \dots, m_s)$ either $n_1\geq n-1$ or $m_1=m$ were already
classified, we shall reduce our investigation to the case of the
characteristic sequence $(n_1, \dots, n_k| m_1, \dots m_s),$
where $n_1\leq n-2$ and $m_1 \leq m-1$

From the Definition \ref{d4} we have that a Leibniz algebra $L_0$
has characteristic sequence $(n_1, \dots, n_k).$ Let $l \in
\mathbb{N}$ be a nilindex of the Leibniz algebra $L_0.$ Since $n_1
\leq n-2,$ then we have $l \leq n-1$ and Leibniz algebra $L_0$ has
at least two generators (the elements which belong to the set
$L_0\setminus L_0^2$).

For the completeness of the statement below we present the
classifications of the papers \cite{FilSup}--\cite{C-G-O-Kh} and
\cite{G-K-N}.

$Leib_{1,m}:$
$$\small
\left\{\begin{array}{l} [y_i,x_1]=y_{i+1}, \ \ 1\leq i \leq m-1.
\end{array}\right.$$
$Leib_{n,1}:$
$$ \small\left\{\begin{array}{ll} [x_i,x_1]=x_{i+1},& 1 \leq i \leq n-1,\\{}
[y_1,y_1]=\alpha x_n, & \alpha = \{0, \ 1\}.\end{array}\right.$$
$Leib_{2,2}:$
$$\small\begin{array}{ll}
\left\{\begin{array}{l} [y_1,x_1]=y_2, \\ {[}x_1,y_1]=\displaystyle \frac12 y_2, \\[2mm] {[}x_2,y_1]=\displaystyle
y_2, \\[2mm] [y_1,x_2] = 2y_2, \\ {[}y_1,y_1]=x_2, \\
\end{array}\right.&
 \left\{\begin{array}{l}
[y_1,x_1]=y_2, \\ {[}x_2,y_1]=\displaystyle y_2, \\[2mm] {[}y_1,x_2]= 2y_2, \\ {[}y_1,y_1]=x_2. \\
\end{array}\right.
\end{array}$$
$Leib_{2,m}, \ m \ \rm{is \ odd}:$
$$\small\begin{array}{ll}
&\\[2mm]
\left\{\begin{array}{ll} [x_1,x_1]=x_2, \ & m\geq 3, \\{}
[y_i,x_1]=y_{i+1},& 1\leq i\leq m-1,\\{} [x_1,y_i]=-y_{i+1},&1\leq
i\leq m-1,\\{} [y_i,y_{m+1-i}]=(-1)^{j+1}x_2, & 1\leq i\leq
\frac{m+1}2.
\end{array}\right.&
\left\{\begin{array}{ll} [y_i,x_1]= [x_1, y_i]= -y_{i+1},& 1\leq
i\leq m-1,\\{} [y_{m+1-i},y_i]=(-1)^{j+1}x_2, & 1\leq i\leq
\frac{m+1}2.
\end{array}\right.
\end{array}$$

In order to present the classification of Leibniz superalgebras
with characteristic sequence $(n-1,1 | m)$, $n \geq 3$ and
nilindex $n+m$ we need to introduce the following families of
superalgebras:
$$\bf Leib_{n,n-1}:$$
$L(\alpha_4, \alpha_5, \ldots, \alpha_n, \theta):$
$$
\left\{\begin{array}{ll}
[x_1,x_1]=x_3,& \\[1mm] {[}x_i,x_1]=x_{i+1},&    2 \le i \le n-1,
\\[1mm]
{[}y_j,x_1]=y_{j+1},&    1 \le j \le n-2,
\\[1mm]
{[}x_1,y_1]= \frac12 y_2,&
\\[1mm]
{[}x_i,y_1]= \frac12 y_i,  &    2 \le i \le n-1,
\\[1mm]
{[}y_1,y_1]=x_1,&
\\[1mm]
{[}y_j,y_1]=x_{j+1},& 2 \le j \le n-1,
\\[1mm]
{[}x_1,x_2]=\alpha_4x_4+ \alpha_5x_5+ \ldots +
\alpha_{n-1}x_{n-1}+ \theta x_n,&
\\[1mm]
{[}x_j,x_2]= \alpha_4x_{j+2}+ \alpha_5x_{j+3}+ \ldots +
\alpha_{n+2-j}x_n,& 2 \le j \le n-2,
\\[1mm]
{[}y_1,x_2]= \alpha_4y_3+ \alpha_5y_4+ \ldots +
\alpha_{n-1}y_{n-2}+\theta y_{n-1},&
\\[1mm]
{[}y_j,x_2]= \alpha_4y_{j+2}+ \alpha_5y_{j+3}+ \ldots +
\alpha_{n+1-j}y_{n-1},& 2 \le j \le n-3. \end{array} \right.$$

$G(\beta_4,\beta_5, \ldots, \beta_n, \gamma):$
$$ \left\{\begin{array}{ll}
[x_1,x_1]=x_3, \\[1mm] {[}x_i,x_1]=x_{i+1},&    3 \le i \le n-1,
\\[1mm]
{[}y_j,x_1]=y_{j+1}, &    1 \le j \le n-2,
\\[1mm]
{[}x_1,x_2]= \beta_4x_4+\beta_5x_5+\ldots+\beta_nx_n,&
\\[1mm]
{[}x_2,x_2]= \gamma x_n,& \\[1mm]
{[}x_j,x_2]= \beta_4x_{j+2}+\beta_5x_{j+3}+\ldots+\beta_{n+2-j}x_n,& 3\le j\le n-2, \\[1mm]
{[}y_1,y_1]=x_1,&
\\[1mm]
{[}y_j,y_1]=x_{j+1},& 2 \le j \le n-1,
\\[1mm]
{[}x_1,y_1]= \frac12 y_2,&
\\[1mm]
{[}x_i,y_1]= \frac12 y_i,& 3\le i\le n-1,
\\[1mm]
{[}y_j,x_2]= \beta_4y_{j+2}+\beta_5y_{j+3}+ \ldots +
\beta_{n+1-j}y_{n-1},& 1\le j\le n-3. \end{array} \right.$$
$$\bf Leib_{n,n}:$$
$M(\alpha_4, \alpha_5, \ldots, \alpha_n, \theta, \tau):$

$$ \left\{
\begin{array}{ll}
[x_1,x_1]=x_3,& \\[1mm] {[}x_i,x_1]=x_{i+1},&     2 \le i \le n-1,
\\[1mm]
{[}y_j,x_1]=y_{j+1}, &    1 \le j \le n-1,
\\[1mm]
{[}x_1,y_1]=  \frac12 y_2,&
\\[1mm]
{[}x_i,y_1]=  \frac12 y_i, &    2 \le i \le n,
\\[1mm]
{[}y_1,y_1]=x_1,&
\\[1mm]
{[}y_j,y_1]=x_{j+1},& 2 \le j \le n-1,
\\[1mm]
{[}x_1,x_2]=\alpha_4x_4+ \alpha_5x_5+ \ldots +
\alpha_{n-1}x_{n-1}+ \theta x_n,&
\\[1mm]
{[}x_2,x_2]=\gamma_4x_4,&\\[1mm]
 {[}x_j,x_2]= \alpha_4x_{j+2}+
\alpha_5x_{j+3}+ \ldots + \alpha_{n+2-j}x_n,&3 \le j \le n-2,
\\[1mm]
{[}y_1,x_2]= \alpha_4y_3+ \alpha_5y_4+ \ldots +
\alpha_{n-1}y_{n-2}+\theta y_{n-1}+\tau y_n,&
\\[1mm]
{[}y_2,x_2]= \alpha_4y_4+ \alpha_5y_4+ \ldots +
\alpha_{n-1}y_{n-1}+\theta y_n,&
\\[1mm]
{[}y_j,x_2]= \alpha_4y_{j+2}+ \alpha_5y_{j+3}+ \ldots +
\alpha_{n+2-j}y_{n},& 3 \le j \le n-2.\end{array} \right.$$

$H(\beta_4, \beta_5, \ldots,\beta_n, \delta , \gamma ):$
$$ \left\{
\begin{array}{ll}
[x_1,x_1]=x_3,& \\[1mm] {[}x_i,x_1]=x_{i+1},&     3 \le i \le n-1,
\\[1mm]
{[}y_j,x_1]=y_{j+1}, &    1 \le j \le n-2,
\\[1mm]
{[}x_1,x_2]= \beta_4x_4+\beta_5x_5+\ldots+\beta_nx_n,&
\\[1mm]
{[}x_2,x_2]= \gamma x_n, &\\[1mm]
{[}x_j,x_2]= \beta_4x_{j+2}+\beta_5x_{j+3}+\ldots+\beta_{n+2-j}x_n,& 3\le j\le n-2, \\[1mm]
{[}y_1,y_1]=x_1,&
\\[1mm]
{[}y_j,y_1]=x_{j+1},& 2 \le j \le n-1,
\\[1mm]
{[}x_1,y_1]= \frac12 y_2,&
\\[1mm]
{[}x_i,y_1]= \frac12 y_i,& 3\le i\le n-1,
\\[1mm]
{[}y_1,x_2]= \beta_4y_3+\beta_5y_4+ \ldots + \beta_ny_{n-1}+\delta y_n,& \\[1mm]
{[}y_j,x_2]= \beta_4y_{j+2}+\beta_5y_{j+3}+ \ldots +
\beta_{n+2-j}y_n,& 2\le j\le n-2. \end{array} \right.$$

Analogously, for the Leibniz superalgebras with characteristic
sequence $(n | m-1,1)$, $n \geq 2$ we introduce the following
families of superalgebras:
$$\bf Leib_{n,n+1}:$$
$E\left( \gamma, \beta_{\left[ \frac{n+4}2\right]}, \beta_{\left[
\frac{n+4}2\right]+1}, \ldots, \beta_n,\beta\right):$

$$ \left\{
\begin{array}{ll}
[x_i,x_1]=x_{i+1},&     1 \le i \le n-1,
\\[1mm]
[y_j,x_1]=y_{j+1}, &    1 \le j \le n-1,
\\[1mm]
[x_i,y_1]=\frac12 y_{i+1}, &1\le i\le n-1,
\\[1mm]
[y_j,y_1]=x_{j}, & 1\le j\le n, \\[1mm]
[y_{n+1},y_{n+1}]=\gamma x_n, & \\[1mm]
[x_i,y_{n+1}]=\sum\limits_{k=\left[\frac{n+4}2\right]}^{n+1-i}
\beta_k y_{k-1+i}, & 1\le i\le \left[ \frac{n-1}2\right],
\\[1mm]
[y_1,y_{n+1}]=-2\sum\limits_{k=\left[\frac{n+4}2\right]}^{n}
\beta_k x_{k-1}+\beta x_n,& \\[1mm]
[y_j,y_{n+1}]=-2\sum\limits_{k=\left[\frac{n+4}2\right]}^{n+2-j}
\beta_k x_{k-2+j},& 2\le j\le \left[\frac{n+1}2\right]. \\[1mm]
\end{array} \right.$$

$$\bf Leib_{n,n+2}:$$
$ F\left( \beta_{\left[\frac{n+5}2\right]},
\beta_{\left[\frac{n+5}2\right]+1}, \ldots,\right.$ $\left.
\beta_{n+1}\right): $

$$ \left\{
\begin{array}{ll}
[x_i,x_1]=x_{i+1},&     1 \le i \le n-1,
\\[1mm]
[y_j,x_1]=y_{j+1}, &    1 \le j \le n,
\\[1mm]
[x_i,y_1]=\frac12 y_{i+1}, &1\le i\le n,
\\[1mm]
[y_j,y_1]=x_{j}, & 1\le j\le n, \\[1mm]
[x_i,y_{n+2}]=\sum\limits_{k=\left[\frac{n+5}2\right]}^{n+2-i}
\beta_k y_{k-1+i}, & 1\le i\le \left[ \frac{n}2\right],
\\[1mm]
[y_j,y_{n+2}]=-2\sum\limits_{k=\left[\frac{n+5}2\right]}^{n+2-j}
\beta_k x_{k-2+j}, & 1\le j\le \left[ \frac{n}2\right]
\end{array} \right.$$

Let us introduce also the following operators which act on
$k$-dimensional vectors:
$$
\begin{array}{rl} j & \\ V^0_{j,k}(\alpha_1, \alpha_2,\ldots, \alpha_k) =
 ( 0, \ldots, 0, 1, & \delta \sqrt[j]{\delta ^{j+1}} S_{m,j}^{j+1} \alpha_{j+1}, \delta \sqrt[j]{\delta
 ^{j+2}}
S_{m,j}^{j+2} \alpha_{j+2}, \ldots , \delta \sqrt[j]{\delta ^{k}}
S_{m,j}^{k} \alpha_{k}   ) ; \\ \end{array}
$$ $$
\begin{array}{rl} j & \\ V^1_{j,k}(\alpha_1, \alpha_2,\ldots, \alpha_k) =
 ( 0, \ldots, 0, 1, &  S_{m,j}^{j+1} \alpha_{j+1},  S_{m,j}^{j+2} \alpha_{j+2}, \ldots , S_{m,j}^{k}
\alpha_{k}  ) ; \\ \end{array}
$$ $$
\begin{array}{rl} j & \\ V^2_{j,k}(\alpha_1, \alpha_2,\ldots, \alpha_k) =
 ( 0, \ldots, 0, 1, &  S_{m,2j+1}^{2(j+1)+1} \alpha_{j+1},  S_{m,2j+1}^{2(j+2)+1} \alpha_{j+2}, \ldots ,
S_{m,2j+1}^{2k+1} \alpha_{k}  ) ; \\ \end{array}
$$ $$
V^0_{k+1,k}(\alpha_1, \alpha_2,\ldots, \alpha_k) =
V^1_{k+1,k}(\alpha_1, \alpha_2,\ldots, \alpha_k) =
V^2_{k+1,k}(\alpha_1, \alpha_2,\ldots, \alpha_k) = (0, 0, \ldots,
0);
$$
$$
W_{s,k}(0,0,\ldots,\stackrel{j-1}{0},\stackrel{j}{1},S_{m,j}^{j+1}\alpha_{j+1},S_{m,j}^{j+2}\alpha_{j+2},\ldots,
S_{m,j}^k\alpha_k,\gamma)=
$$
$$
=( 0, 0,\ldots, \stackrel{j}{1} ,0,\ldots,\stackrel{s+j}{1},
S_{m,s}^{s+1}\alpha_{s+j+1}, S_{m,s}^{s+2}\alpha_{s+j+2},\ldots,
S_{m,s}^{k-j}\alpha_k, S_{m,s}^{k+6-2j}\gamma),
$$
$$
W_{k+1-j,k}(0,0,\ldots,\stackrel{j-1}{0},\stackrel{j}{1},S_{m,j}^{j+1}\alpha_{j+1},S_{m,j}^{j+2}\alpha_{j+2},\ldots,
S_{m,j}^k\alpha_k,\gamma)=$$ $\qquad
=(0,0,\ldots,\stackrel{j}{1},0,\ldots,1),$ $$
W_{k+2-j,k}(0,0,\ldots,\stackrel{j-1}{0},\stackrel{j}{1},S_{m,j}^{j+1}
\alpha_{j+1},S_{m,j}^{j+2}\alpha_{j+2},\ldots,
S_{m,j}^k\alpha_k,\gamma)=$$ $\qquad
=(0,0,\ldots,\stackrel{j}{1},0,\ldots,0), $ \\
where $k\in N,$ $\delta=\pm 1,$ $1\le j\le k,$ $1\le s\le k-j,$
and $\displaystyle S_{m,t}=\cos\frac{2\pi m}t+i\sin\frac{2\pi m}t$
$(m=0,1,\ldots, t-1).$

Below we present the complete list of pairwise non-isomorphic
Leibniz superalgebras with $n+m:$

with characteristic sequence equal to $(n-1,1|m):$
$$
\begin{array}{l} L\left( V^1_{j,n-3}\left( \alpha_4,\alpha_5,\ldots,
\alpha_n\right),S_{m,j}^{n-3}\theta\right),\qquad \ \
1\le j\le n-3, \\[2mm] L(0,0,\ldots,0,1), \ L(0,0,\ldots,0), \ G(0,0,\ldots,0,1), \
G(0,0,\ldots,0),
\\[2mm] G\left( W_{s,n-2}\left( V^1_{j,n-3}\left(
\beta_4,\beta_5,\ldots,\beta_n\right),\gamma\right)\right),\quad
1\le j\le n-3,\ 1\le s\le n-j,
\\[2mm] M\left( V^1_{j,n-2}\left( \alpha_4,\alpha_5,\ldots,\alpha_n\right),S_{m,j}^{n-3}\theta\right),
\qquad \ 1\le j\le n-2,
\\[2mm] M(0,0,\ldots,0,1), \ M(0,0,\ldots,0), \ H(0,0,\ldots,0,1), \
H(0,0,\ldots,0), \\[2mm] H\left( W_{s,n-1}\left( V^1_{j,n-2}\left(
\beta_4,\beta_5,\ldots,\beta_n\right),\gamma\right)\right),\quad
1\le j\le n-2,\ 1\le s\le n+1-j, \\\end{array}
$$ with characteristic sequence equal to $(n|m-1,1)$ if $n$ is odd
(i.e. $n=2q-1$):
$$
\begin{array}{lll}
E\left(1,\delta\beta_{q+1}, V_{j,q-2}^0(\beta_{q+2}, \beta_{q+3},
\ldots, \beta_n),0\right), & \displaystyle
\beta_{q+1}\ne \pm\frac12, & 1\le j\le q-1, \\[2mm]
E\left(1,\beta_{q+1}, V_{j,q-1}^0(\beta_{q+2}, \beta_{q+3},
\ldots,
\beta_n,\beta)\right), & \beta_{q+1}=\displaystyle  \pm\frac12, & 1\le j\le q, \\[2mm] E(0,1,V_{j,q-2}^0(\beta_{q+2}, \beta_{q+3},
\ldots, \beta_n),0), & 1\le j\le q-1, & \\[2mm] E(0,0, W_{s,q-1}(V^1_{j,q-1}(\beta_{q+2}, \beta_{q+3}, \ldots, \beta_n,
\beta))), & 1\le j\le q-1, & 1\le s\le q-j, \\[2mm] E(0,0,\ldots,0); \\
\end{array}
$$

if $n$ is even (i.e. $n=2q$):
$$
\begin{array}{lll} E(1,V^2_{j,q-1}(\beta_{q+2}, \beta_{q+3}, \ldots, \beta_n,), 0), & 1\le j\le q, & \\[2mm]
E(0, W_{s,q}(V^1_{j,q}(\beta_{q+2}, \beta_{q+3}, \ldots,
\beta_n,\beta))), & 1\le j\le q, & 1\le s\le q+1-j,
\\[2mm] E(0,0,\ldots,0). \\
\end{array}
$$
$$
F\left( W_{s,n+2-\left[\frac{n+5}2\right]} \left( V^1_{j,
n+2-\left[\frac{n+5}2\right]} \left( \beta_{\left[ \frac{n+5}2
\right]}, \beta_{\left[\frac{n+5}2\right]+1}, \ldots,
\beta_{n+1}\right)\right)\right),
$$
where $1\le j\le n+2-\displaystyle \left[\frac{n+5}2\right],$
$1\le s\le n+3-\displaystyle \left[ \frac{n+5}2\right]-j,$
$$ F(0,0,\ldots,0). $$

For a given Leibniz algebra $A$ of the nilindex $l$ we put
$gr(A)_i = A^i / A^{i+1}, \quad 1 \leq i \leq l-1$ and $gr(A) =
gr(A)_1 \oplus gr(A)_2 \oplus \dots \oplus gr(A)_{l-1}.$ Then
$[gr(A)_i, gr(A)_j] \subseteq gr(A)_{i+j}$ and we obtain the
graded algebra $gr(A).$

\begin{defn} \label{d5} The gradation constructed in this way is called the natural
gradation and if a Leibniz algebra $G$ is isomorphic to $gr(A)$ we
say that the algebra $G$ is naturally graded Leibniz algebra.
\end{defn}

\section{The main result}
Let $L$ be a Leibniz superalgebra with characteristic sequence
$(n_1, \dots, n_k| m_1, \dots, m_s),$ where $ n_1 \leq n-2,$ $m_1
\leq m-1$ and of nilindex $n+m.$ Since the second part of the
characteristic sequence of the Leibniz superalgebra $L$ is equal
to $(m_1, \dots, m_s)$ then by the definition of the
characteristic sequence there exists a nilpotent endomorphism
$R_x$ ($x\in L_0\setminus L_0^2$) of the space $L_1$ such that its
Jordan form consists of $s$ Jordan blocks. Therefore, we can
assume the existence of an adapted basis $\{y_1, y_2, \dots,
y_m\}$ of the subspace $L_1,$ such that
$$
\left\{\begin{array}{ll} [y_j,x]=y_{j+1}, \ & j \notin \{m_1,
m_1+m_2, \dots, m_1+m_2+ \dots + m_s\},\\{} [y_j,x]=0,& j \in
\{m_1, m_1+m_2, \dots, m_1+m_2+ \dots + m_s\}.
\end{array}\right. \eqno(1)
$$ for some $x \in L_0\setminus L_0^2.$

Further we shall use a homogeneous basis $\{x_1, \dots, x_n\}$
with respect to natural gradation of the Leibniz algebra $L_0,$
which is also agreed with the lower central sequence of $L.$

The main result of the paper is that the nilindex of the Leibniz
superalgebra $L$ with characteristic sequence $(n_1, \dots, n_k|
m_1, \dots, m_s),$ $ n_1 \leq n-2, m_1 \leq m-1$ is less than
$n+m.$

According to the Theorem \ref{t1} we have the description of
single-generated Leibniz superalgebras, which have nilindex
$n+m+1.$ If the number of generators is greater than two, then
superalgebra has nilindex less than $n+m.$ Therefore, we should
consider case of two-generated superalgebras.

The possible cases for the generators are:

1. Both generators lie in $L_0,$ i.e. $dim(L^2)_0 = n-2$ and
$dim(L^2)_1 = m;$

2. One generator lies in $L_0$ and another one lies in $L_1,$
 i.e. $dim(L^2)_0 = n-1$ and $dim(L^2)_1 = m-1;$

3. Both generators lie in $L_1,$ i.e. $dim(L^2)_0 = n$ and
$dim(L^2)_1 = m-2.$

Moreover, two-generated superalgebra $L$ has nilindex $n+m$ if and
only if $dim L^k = n+m-k$ for $2 \leq k \leq n+m.$

Since $m\neq 0$ we omit the case where both generators lie in even
part.

\subsection{The case of one generator in $L_0$ and another one in $L_1$}

\

Since $dim(L^2)_0 = n-1$ and $dim(L^2)_1 = m-1$ then there exist
some $m_j,$ $0 \leq j \leq s-1$ (here we assume $m_0=0$) such that
$y_{m_1+ \dots + m_j + 1} \notin L^2.$ By a shifting of basic
elements we can assume that $m_j=m_0,$ i.e. the basic element
$y_1$ can be chosen as a generator of the superalgebra $L.$ Of
course, by this shifting the condition from definition of the
characteristic sequence $m_1 \geq m_2 \geq \dots \geq m_s$ can be
broken, but further we shall not use the condition.

Let $L=L_0 \oplus L_1$ be a two generated Leibniz superalgebra
from $Leib_{n,m}$ with characteristic sequence equal to $(n_1,
\dots, n_k| m_1, \dots, m_s)$ and let $\{x_1, \dots, x_n, y_1,
\dots, y_m\}$ be a basis of the $L.$

\begin{lem}\label{l1} Let one generator lies in $L_0$ and
another one lies in $L_1.$ Then $x_1$ and $y_1$ can be chosen as
generators of the $L.$ Moreover, in equality (1) instead of
element $x$ we can suppose $x_1.$
\end{lem}
\begin{proof}
As mentioned above $y_1$ can be chosen as the first generator of
$L$. If $x\in L\setminus L^2$ then the assertion of the lemma is
evident. If $x\in L^2$ then there exists some $i_0$ ($2\leq i_0$)
such that $x_{i_0}\in L\setminus L^2.$ Set $x_1'=Ax + x_{i_0}$ for
$A\neq 0$ then $x_1'$ is a generator of the superalgebra $L$
(since $x_1'\in L\setminus L^2$). Moreover, making transformation
of the basis of $L_1$ as follows
$$
\left\{\begin{array}{ll} y_j'= y_j, \ & j \in \{1, m_1+1, \dots,
m_1+m_2+ \dots + m_{s-1}+1\},\\{} y_j'= [y_{j-1}',x_1'],& j \notin
\{1, m_1+1, \dots, m_1+m_2+ \dots + m_{s-1}+1\}.
\end{array}\right.$$
and taking sufficiently big value of the parameter $A$ we preserve
the equality (1). Thus, in the basis $\{x_1', x_2, \dots, x_n,
y_1', y_2', \dots, y_m'\}$ the elements $x_1'$ and $y_1'$ are
generators.
\end{proof}

Due to Lemma \ref{l1} further we shall suppose that $\{x_1, y_1\}$
are generators of the Leibniz superalgebra $L.$ Therefore,
$$L^2 = \{x_2, x_3,  \dots, x_n, y_2,
y_3, \dots, y_m\}.$$ Let us introduce the notations:
$$[x_i,y_1]= \sum\limits_{j=2}^m \alpha_{i,j}y_j,\ 1 \le i \le n, \  \ \ [y_i,y_1]=
\sum\limits_{j=2}^n \beta_{i,j}x_j, \ 1 \le i \le m. \eqno (2)$$

Without loss of generality we can assume that
$y_{m_1+\dots+m_i+1}\in L^{t_i}\setminus L^{{t_i}+1}$, where
$t_i<t_j$ for $1\leq i<j\leq s-1.$

Firstly we consider the case of $dim(L^3)_0 = n-1,$ then
$dim(L^3)_0 = n-2.$\\[0,5mm]

 {\bf Case $dim(L^3)_0 = n-1$}.\\[0,5mm]

In this subcase we have
$$L^3 = \{x_2, x_3, \dots, x_n, y_3, \dots,
y_{m_1}, B_1y_2 +B_2y_{m_1+1}, y_{m_1+2}, \dots, y_m\},$$ where
$(B_1,B_2)\neq(0,0).$

Analyzing the way the element $x_2$ can be obtained, we conclude
that there exist $i_0 \ (2 \leq i_0 \leq m)$ such that
$[y_{i_0},y_1]= \sum\limits_{j=2}^n \beta_{i_0,2}x_j, \
\beta_{i_0,2} \neq 0.$

Let us show that $i_0 \notin\{m_1+1, \dots, m_1+ \dots +
m_{s-1}+1\}.$ It is known that the elements $y_{m_1+ m_2+1},
\dots, y_{m_1+ \dots + m_{s-1}+1}$ are generated from the products
$[x_i, y_1], \  (2 \leq i \leq n).$ Due to nilpotency of $L$ we
get $i_0 \notin \{m_1+m_2+1, \dots, m_1+ \dots + m_{s-1}+1\}.$ If
$y_{m_1+1}$ is generated by $[x_1, y_1],$ i.e. in the expression
$[x_1, y_1] = \sum\limits_{j=2}^m \alpha_{1,j}y_j$ $\alpha_{1,
m_1+1} \neq 0$ then we consider the product
$$[[x_1, y_1],y_1] = [\sum\limits_{j=2}^m \alpha_{1,j}y_j, y_1]
= \alpha_{1, m_1+1}\beta_{m_1+1,2}x_2 + \sum\limits_{i \geq 3}
(*)x_i.$$ On the other hand,
$$[[x_1, y_1],y_1] = \frac 1 2 [x_1,[y_1, y_1]] =
\frac 1 2 [x_1, \sum\limits_{j=2}^n \beta_{1,j}x_j]
 = \sum\limits_{i \geq 3}
(*)x_i$$

Comparing the coefficients at the corresponding basic elements we
obtain $\alpha_{1, m_1+1}\beta_{m_1+1,2}=0,$ which implies
$\beta_{m_1+1,2}=0.$ It means that $i_0 \neq m_1+1.$ Therefore,
$\beta_{i_0,2} \neq 0,$ where $i_0 \notin\{m_1+1, \dots, m_1+
\dots + m_{s-1}+1\}.$

\

\textbf{Case $y_2 \notin L^3.$} Then $B_2 \neq 0.$ Let
$h\in\mathbb{N}$ be a number such that $x_2 \in L^h\setminus
L^{h+1},$ that is $$ L^h = \{x_2, x_3, \dots, x_n, y_h,\dots,
y_{m_1}, B_1y_2 +B_2y_{m_1+1}, y_{m_1+2}, \dots, y_m\}, \  h \geq
3,$$
$$L^{h+1} = \{x_3, x_4, \dots, x_n, y_h, \dots,
y_{m_1}, B_1y_2 +B_2y_{m_1+1}, y_{m_1+2},\dots, y_m\}.$$

Since the elements $B_1y_2 +B_2y_{m_1+1}, y_{m_1+ m_2+1}, \dots,
y_{m_1+ \dots +m_{s-1}+1}$ are generated from the multiplications
$[x_i, y_1], 2 \leq i \leq n$ it follows that $h \leq m_1 +1.$

So, $x_2$ can be obtained only from product $[y_{h-1},y_1]$ and
thereby $\beta_{h-1,2} \neq 0.$ Making the change $x_2'=
\sum\limits_{j=2}^n \beta_{h-1,j}x_j$ we can assume that
$[y_{h-1}, y_1] = x_2.$

 Let now $p$ is a number such that $y_h \in L^{h+p}\setminus L^{h+p+1}.$
Then for the powers of superalgebra $L$ we have the following
$$L^{h+p} = \{x_{p+2}, x_{p+3}, \dots, x_n, y_h, \dots,
y_{m_1}, B_1y_2 +B_2y_{m_1+1}, y_{m_1+2}, \dots, y_m\}, \ p \geq
1,$$
$$L^{h+p+1} = \{x_{p+2}, x_{p+3}, \dots, x_n, y_{h+1}, \dots,
y_{m_1}, B_1y_2 +B_2y_{m_1+1}, y_{m_1+2}, \dots, y_m\}.$$

In the following lemma the useful expression for the products
$[y_i, y_j]$ is presented.
\begin{lem}\label{le2} The equality:
$$[y_i, y_j] = (-1)^{h-1-i}C_{j-1}^{h-1-i}x_{i+j+2-h} +
\sum\limits_{t > i+j+2-h}(*)x_t,  \eqno (3)$$ $1 \leq i \leq h-1,
\  h-i \leq j \leq min\{h-1, h-1+p-i\},$ holds.
\end{lem}
\begin{proof} The proof is deduced by the induction on $j$ at any
value of $i.$
\end{proof}

For the natural number $p$ we have the following
\begin{lem}\label{l3} Under the above conditions
$p=1.$
\end{lem}
\begin{proof}
Assume the contrary, i.e. $p> 1.$ Then we can suppose
$$[x_i, x_1] = x_{i+1}, \ 2 \leq  i \leq p, \quad [x_{p+1}, y_1] =
\sum\limits_{j = h}^m\alpha_{p+1,j} y_j, \quad \alpha_{p+1,h} \neq
0.$$

Using the equality (3) we consider the following chain of
equalities

$$[y_1, [y_{h-1},x_1]] = [[y_1, y_{h-1}],x_1] - [[y_1,x_1],
y_{h-1}]= (-1)^{h-2}x_3 + \sum\limits_{t \geq 4}(*)x_t-$$
$$-(-1)^{h-3}(h-2)x_3 + \sum\limits_{t \geq 4}(*)x_t=
(-1)^{h}(h-1)x_3 + \sum\limits_{t \geq 4}(*)x_t.$$

If $h \leq m_1,$ then  $[y_1, [y_{h-1},x_1]] = [y_1, y_h]$. Since
$y_h \in L^{h+p}$ and $p > 1$ then in the decomposition of $[y_1,
y_h]$ the coefficient at the basic elements $x_2$ and $x_3$ are
equal to zero. Therefore, from the above equalities we get a
contradiction with assumption $p>1.$

If $h = m_1 +1,$ then $[y_1, [y_{h-1},x_1]] = 0$ and we also
obtain the irregular equality $(-1)^{h}(h-1)x_3 + \sum\limits_{t
\geq 4}(*)x_t=0.$ Therefore, the proof of the lemma is completed.
\end{proof}
We resume our main result in considered cases in the following

\begin{thm}\label{t2}
Let $L=L_0 \oplus L_1$  be a Leibniz superalgebra from
$Leib_{n,m}$ with characteristic sequence equal to $(n_1, \dots,
n_k | m_1, \dots, m_s),$ where $n_1\leq n-2, \ m_1\leq m-1$ and
let $dim(L^3)_0 = n-1$ with $y_2 \notin L^3.$ Then $L$ has a
nilindex less than $n+m.$
\end{thm}

\begin{proof} Let us assume the contrary, i.e. nilindex of the superalgebra $L$ equal
to $n+m.$ Then according to the Lemma \ref{l3} we have
$$L^{h+2} = \{x_3, \dots, x_n, y_{h+1},\dots,
y_{m_1}, B_1y_2 +B_2y_{m_1+1}, y_{m_1+2}, \dots, y_m\}.$$

Since $y_h \notin L^{h+2},$ it follows that $$ \alpha_{2,h}\neq 0,
\quad \alpha_{i,h}=0 \quad \mbox{for}\quad  i>2.$$

Consider the product $$[[y_{h-1}, y_1], y_1] = \frac 1 2 [
y_{h-1}, [ y_1, y_1]] = \frac 1 2 [y_{h-1},
\sum\limits_{i=2}^n\beta_{1,i}x_i] .$$

The element $y_{h-1}$ belongs to $L^{h-1}$ and elements $x_2, x_3,
\dots, x_n$ lie in $L^3.$ Hence $\frac 1 2 [y_{h-1},
\sum\limits_{i=2}^n\beta_{1,i}x_i] \in L^{h+2}.$ Since $y_h \notin
L^{h+2},$  we obtain that $[[y_{h-1}, y_1], y_1]
=\sum\limits_{j\geq h+1}(*)y_j.$

On the other hand,
$$[[y_{h-1}, y_1], y_1] = [x_2, y_1]
= \alpha_{2,h}y_h
 + \sum\limits_{j=  h+1}^m\alpha_{2,j}y_j.$$

Comparing the coefficients at the basic elements we obtain $
\alpha_{2,h}=0,$ which is a contradiction with the assumption that
the superalgebra $L$ has nilindex equal to $n+m$ and therefore the
assertion of the theorem is proved.
\end{proof}

\textbf{Case $y_2 \in L^3.$} Then $B_2=0$ and the following
theorem is true.

\begin{thm}\label{t3}
Let $L=L_0 \oplus L_1$  be a Leibniz superalgebra from
$Leib_{n,m}$ with characteristic sequence equal to $(n_1, \dots,
n_k | m_1, \dots, m_s),$ where $n_1\leq n-2, \ m_1\leq m-1$ and
let $dim(L^3)_0 = n-1$ with $y_2 \in L^3.$ Then $L$ has a nilindex
less than $n+m.$
\end{thm}
\begin{proof}
We shall prove the assertion of the theorem by contrary method,
i.e. we assume that nilindex of the superalgebra $L$ equal to
$n+m.$ The condition $y_2 \in L^3$ implies
$$L^3 = \{x_2, x_3, \dots, x_n, y_2, \dots,
y_{m_1}, y_{m_1+2}, \dots, y_m\}.$$
 Then $\alpha_{1, m_1+1} \neq 0$ and $\alpha_{i, m_1+1} =
 0$ for $i \geq 2.$
The element $y_2$ is generated from products $[x_i, y_1],$ $i\geq
2$ which implies $y_2 \in L^4.$ Since $[y_{m_{1} +1},
y_1]=[[x_1,y_1],y_1]=\frac{1}{2}[x_1,[y_1,y_1]]=\frac{1}{2}[x_1,\sum(*)x_i]$
and $x_2$ is a generator of the Leibniz algebra $L_0$ then $x_2$
can not generated from the product $[y_{m_{1} +1}, y_1].$ Thereby
$x_2$ also belongs to $L^4.$

Consider the equality
$$[[x_1, y_1], x_1] = [x_1,[ y_1, x_1]] + [[x_1, x_1], y_1] = [x_1, y_2]
-[\sum\limits_{i\geq 3}(*)x_i, y_1].$$

From this it follows that the product $[[x_1, y_1], x_1]$ belongs
to $L^5$ (and therefore belongs to $L^4$).

On the other hand,
$$[[x_1, y_1], x_1] =  [\sum\limits_{j=2}^m \alpha_{1,j} y_j, x_1] =
\alpha_{1,2}y_3 + \dots + \alpha_{1,m_1-1}y_{m_1}
+\alpha_{1,m_1+1}y_{m_1+2} + \dots +\alpha_{1,m-1}y_{m}.$$

Since $\alpha_{1, m_1+1} \neq 0,$ we obtain that $y_{m_1+2} \in
L^4.$ Thus, we have $L^4 = \{x_2, x_3, \dots, x_n, \\ y_2, \dots,
y_{m_1}, y_{m_1+2}, \dots, y_m\},$ that is $L^4 = L^3.$ It is a
contradiction to nilpotency of the superalgebra $L.$

Thus, we get a contradiction with assumption that the superalgebra
$L$ has nilindex equal to $n+m$ and therefore the assertion of the
theorem is proved.
\end{proof}

From Theorems \ref{t2} and \ref{t3} we obtain that Leibniz
superalgebra $L$ with condition $dim (L^3)_0 = n-1$ has nilindex
less than $n+m.$

The investigation of the Leibniz superalgebra with property $dim
(L^3)_0 = n-2$ shows that the restriction to nilindex depends on
the structure of the Leibniz algebra $L_0.$ Below we present some
necessary remarks on nilpotent Leibniz algebras.

Let $A = \{z_1, z_2, \dots, z_n\}$ be an $n$-dimensional nilpotent
Leibniz algebra of nilindex $l$ ($l < n$). Note that algebra $A$
is not single-generated.

\begin{prop} \label{c1} \cite{C-G-O-Kh1}
Let $gr(A)$ be a naturally graded non-Lie Leibniz algebra. Then
$dim A^3  \leq n-4.$ \end{prop}

The result on nilindex of the superalgebra under the condition
$dim(L^3)_0 = n-2$ is established in the following two theorems.

\begin{thm}\label{t4} Let $L=L_0 \oplus L_1$ be a Leibniz superalgebra from $Leib_{n,m}$
with characteristic sequence $(n_1, \dots, n_k | m_1, \dots m_s),$
where $n_1\leq n-2, \ m_1\leq m-1,$ $dim(L^3)_0 = n-2$ and $dim
L_0^3 \leq n-4.$ Then $L$ has a nilindex less than $n+m.$
\end{thm}
\begin{proof}
Let us assume the contrary, i.e. the nilindex of the superalgebra
$L$ is equal to $n+m.$ According to the condition $dim(L^3)_0 =
n-2$ we have
 $$ L^3
= \{x_3, x_4, \dots, x_n, y_2, y_3, \dots, y_m\}.$$

From the condition $dim L_0^3 \leq n-4$ it follows that there
exist at least two basic elements, that do not belong to $L_0^3.$
Without loss of generality, one can assume $x_3, x_4 \notin
L_0^3.$

 Let $h$ be a natural number
such that $x_3 \in L^{h+1}\setminus L^{h+2},$ then we have
 $$L^{h+1} = \{x_3, x_4, \dots, x_n, y_h, y_{h+1}, \dots,
y_m\}, \ h \geq 2, \ \beta_{h-1, 3} \neq 0.$$
$$L^{h+2} = \{x_4, \dots, x_n, y_h, y_{h+1},
\dots, y_m\}.$$

Let us suppose $x_3 \notin L_0^2.$ Then we have that $x_3$ can not
be obtained by the products $[x_i, x_1],$ with $2 \leq i\leq n.$
Therefore, it is generated by products $[y_j, y_1], 2 \leq j \leq
m,$ which implies $h \geq 3$ and $\alpha_{2,2}\neq 0.$

If $h=3,$ then $\beta_{2,3}\neq 0.$

Consider the chain of equalities
$$[[x_2, y_1], y_1] = [\sum \limits_{j=2}^m \alpha_{2,j}y_j, y_1]
=  \sum \limits_{j=2}^m \alpha_{2,j}[y_j, y_1]=
\alpha_{2,2}\beta_{2,3}x_3 + \sum \limits_{i\geq 4}(*)x_i.$$

On the other hand,
$$[[x_2, y_1], y_1] = \frac 1 2 [x_2, [ y_1, y_1]] =  \frac 1 2 [x_2,
\sum \limits_{i=2}^n\beta_{1,i}x_i] = \frac 1 2 \sum
\limits_{i=2}^n\beta_{1,i} [x_2, x_i] =  \sum \limits_{i\geq
4}(*)x_i.$$

Comparing the coefficients at the corresponding basic elements, we
get a contradiction with $ \beta_{2,3} = 0.$ Thus, $h \geq 4.$

Since $y_2 \in L^3$ and $h \geq 4$ we have $y_{h-2} \in L^{h-1},$
which implies $[y_{h-2}, y_2] \in L^{h+2} = \{x_4, \dots, x_n,
y_h, y_{h+1}, \dots, y_m\}.$ It means that in the decomposition
$[y_{h-2}, y_2]$ the coefficient at the basic element $x_3$ is
equal to zero.

On the other hand,
$$[y_{h-2}, y_2] = [y_{h-2}, [y_1, x_1]] = [[y_{h-2},
y_1], x_1] - [[y_{h-2}, x_1], y_1] =$$ $$=[ \sum\limits_{i=2}^n
\beta_{h-2,i}x_i, x_1] - [y_{h-1}, y_1]= - \beta_{h-1,3}x_3 +
\sum\limits_{i\geq 4}(*)x_i.$$

Hence, we get $\beta_{h-1,3} = 0,$ which is obtained from the
assumption $x_3 \notin L_0^2.$

Therefore, we have $x_3, x_4 \in L_0^2\setminus L_0^3.$ The
condition $x_4 \notin L_0^3$ deduce that $x_4$ can not be obtained
by the products $[x_i, x_1],$ with $3 \leq i\leq n.$ Therefore, it
is generated by products $[y_j, y_1], h \leq j \leq m.$ Hence,
$L^{h+3} = \{x_4, \dots, x_n, y_{h+1}, \dots, y_m\}$ and $y_h \in
L^{h+2} \setminus L^{h+3},$ which implies $\alpha_{3,h} \neq 0.$

Let $p$ ($ p \geq 3$) be a natural number such that $x_4 \in
L^{h+p} \setminus L^{h+p+1}.$

Suppose that $p=3.$ Then $\beta_{h,4}\neq 0.$

Consider the chain of equalities
$$[[x_3, y_1], y_1] = [\sum \limits_{j=h}^m \alpha_{3,j}y_j, y_1]
=  \sum \limits_{j=h}^m \alpha_{3,j}[y_j, y_1]=
\alpha_{3,h}\beta_{h,4}x_4 + \sum \limits_{i\geq 5}(*)x_i.$$

On the other hand,
$$[[x_3, y_1], y_1] = \frac 1 2 [x_3, [ y_1, y_1]] =  \frac 1 2[x_3,
\sum \limits_{i=2}^n\beta_{1,i}x_i] =  \frac 1 2 \sum
\limits_{i=2}^n\beta_{1,i} [x_3, x_i] = \sum \limits_{i\geq
5}(*)x_i.$$

Comparing the coefficients at the corresponding basic elements in
these equations we get $\alpha_{3,h}\beta_{h,4} = 0,$ which
implies $ \beta_{h,4} = 0.$ It is a contradiction with assumption
$p=3.$ Therefore, $p\geq 4$ and for the powers of descending lower
sequences we have
$$L^{h+p-2} = \{x_4, \dots, x_n,
y_{h+p-4}, \dots, y_m\},$$
$$L^{h+p-1} = \{x_4, \dots, x_n,
y_{h+p-3}, \dots, y_m\},$$ $$L^{h+p} = \{x_4, \dots, x_n,
y_{h+p-2}, \dots, y_m\},$$ $$L^{h+p+1} = \{x_5, \dots, x_n,
y_{h+p-2}, \dots, y_m\}.$$

It is easy to see that in the decomposition $[y_{h+p-3}, y_1] =
\sum\limits_{i=4}^n \beta_{h+p-3,i}x_i$ we have $\beta_{h+p-3,4}
\neq 0.$

Consider the equalities
$$[y_{h+p-4}, y_2] = [y_{h+p-4}, [y_1, x_1]] = [[y_{h+p-4},
y_1], x_1] - [[y_{h+p-4}, x_1], y_1] =$$ $$=[ \sum\limits_{i=4}^n
\beta_{h+p-3,i}x_i, x_1] - [y_{h+p-3}, y_1]= - \beta_{h+p-3,4}x_4
+ \sum\limits_{i\geq 5}(*)x_i.$$

Since $y_{h+p-4} \in L^{h+p-2}, y_2 \in L^3$ and $\beta_{h+p-3,4}
\neq 0,$ then the element $x_4$ should lie in $L^{h+p+1},$ but it
contradicts to $L^{h+p+1} = \{x_5, \dots, x_n, y_{h+p-2}, \dots,
y_m\}.$ Thus, the superalgebra $L$ has a nilindex less than $n+m.$
\end{proof}

From Theorem \ref{t4} we conclude that Leibniz superalgebra $L=
L_0 \oplus L_1$ with the characteristic sequence $(n_1, \dots,
n_k| m_1, \dots, m_s),$ where $n_1\leq n-2, \ m_1\leq m-1$ and
nilindex $n+m$ can appear only if $dim L_0^3 \geq n-3.$ Taking
into account the condition $n_1 \leq n-2$ and properties of
naturally graded subspaces $gr(L_0)_1,$ $gr(L_0)_2$ we get $dim
L_0^3  = n-3.$

Let $dim L_0^3 =n-3.$ Then $$gr(L_0)_1 = \{\overline{x}_1,
\overline{x}_2\}, \ gr(L_0)_2 = \{\overline{x}_3\}.$$

From Proposition \ref{c1} the naturally graded Leibniz algebra
$gr(L_0)$ is a Lie algebra, i.e. the following multiplication
rules hold
$$
\left\{ \begin{array}{l} [\overline{x}_1,\overline{x}_1]=0, \\{}
[\overline{x}_2,\overline{x}_1]=\overline{x}_3,   \\{}
 [\overline{x}_1,\overline{x}_2]=-\overline{x}_3, \\ {}[\overline{x}_2,\overline{x}_2]=0.
\end{array}\right.
$$

Using these products for the corresponding products in the Leibniz
algebra $L_0$ with the basis $\{x_1, x_2, \dots, x_n\}$ we have
$$
\left\{ \begin{array}{l} [x_1,x_1]=\gamma_{1,4}x_4 +
\gamma_{1,5}x_5 + \dots + \gamma_{1,n}x_n,
\\{} [x_2,x_1]=x_3, \\{} [x_1,x_2]=-x_3 + \gamma_{2,4}x_4 + \gamma_{2,5}x_5 + \dots
+ \gamma_{2,n}x_n, \\ {}[x_2,x_2]=\gamma_{3,4}x_4 +
\gamma_{3,5}x_5 + \dots + \gamma_{3,n}x_n.
\end{array} \right. \eqno(4)$$
\begin{thm}\label{t5}
Let $L=L_0\oplus L_1$ be a Leibniz superalgebra from $Leib_{n,m}$
with characteristic sequence $(n_1,  \dots, n_k | m_1, \dots,
m_s),$ where $n_1\leq n-2, \ m_1\leq m-1,$ $dim (L^3)_0=n-2$ and
$dimL_0^3=n-3.$ Then $L$ has a nilindex less than $n+m.$
\end{thm}

\begin{proof}
Let us suppose the contrary, i.e. the nilindex of the superalgebra
$L$ equals $n+m.$ Then from the condition $dim (L^3)_0=n-2$ we
obtain
$$ L^2 = \{x_2, x_3,\dots, x_n, y_2, \dots, y_m\},$$
$$L^3 = \{x_3, x_4, \dots, x_n, y_2, \dots, y_m\}.$$
$$L^4 \supset \{x_4, \dots, x_n, y_3, \dots, y_{m_1}, B_1y_2 + B_2y_{m_1+1}, y_{m_1+2}  \dots, y_m\}, \quad (B_1, B_2) \neq (0,0).$$

Suppose $x_3 \notin L^4.$ Then
$$L^4 = \{x_4, \dots, x_n, y_2, \dots, y_{m_1}, y_{m_1+1}, \dots, y_m\}.$$
Let $B'_1y_2 + B'_2y_{m_1+1}$ be an element which earlier
disappear in the descending lower sequence for $L$. Then this
element can not to be generated from the products $[x_i, y_1], \ 2
\leq i \leq n.$ Indeed, since $x_3 \notin L^4,$ the element can
not to be generated from $[x_2, y_1].$ Due to structure of $L_0$
the elements $x_i, (3 \leq i \leq n)$ are in $L_0^2,$ i.e. they
are generated by the linear combinations of the products of
elements from $L_0.$ The equalities
$$[[x_i, x_j], y_1] =  [x_i,[ x_j, y_1]] + [[x_i, y_1], x_j] =
[x_i,\sum\limits_{t=2}^m \alpha_{j,t}y_t] + [\sum\limits_{t=2}^m
\alpha_{i,t}y_t, x_i]$$ derive that the element $B'_1y_2 +
B'_2y_{m_1+1}$ can not be obtained by the products $[x_i, y_1], 3
\leq i \leq n.$ However, it means that $x_3\in L^4.$ Thus, we have
$$L^4 = \{x_3, x_4, \dots, x_n, y_3, \dots, y_{m_1}, B_1y_2 + B_2y_{m_1+1}, y_{m_1+2},  \dots, y_m\},$$ where $(B_1, B_2) \neq
(0,0)$ and $B_1B'_2 - B_2B'_1 \neq 0.$

 The simple analysis of descending lower sequences
$L^3$ and $L^4$ implies  $$[x_2, y_1] = \alpha'_{2,2}(B'_1y_2 +
B'_2y_{m_1+1}) + \alpha'_{2,m_1+1}(B_1y_2 + B_2y_{m_1+1})+
\sum\limits_{\begin{array}{c}j=3\\j\neq m_1+1 \end{array}} ^m
\alpha_{2,j}y_j,\quad \alpha'_{2,2} \neq 0.$$

Let $h$ be a natural number such that $x_3 \in L^{h+1}\setminus
L^{h+2},$ i.e.

$$L^h = \{x_3, x_4, \dots, x_n, y_{h-1}, y_h, \dots, y_{m_1}, B_1y_2 + B_2y_{m_1+1}, y_{m_1+2}, \dots, y_m\}, h
\geq 3,$$ $$L^{h+1} = \{x_3, x_4, \dots, x_n, y_h, y_{h+1},\dots,
y_{m_1}, B_1y_2 + B_2y_{m_1+1}, y_{m_1+2}, \dots, y_m\},$$
$$L^{h+2} = \{x_4, \dots, x_n, y_h, y_{h+1},\dots, y_{m_1}, B_1y_2 + B_2y_{m_1+1}, y_{m_1+2},  \dots, y_m\}.$$

If $h=3,$ then $[B'_1y_2 + B'_2y_{m_1+1}, y_1] = \beta'_{2,3}x_3 +
\sum\limits_{i\geq 4}(*)x_4,$ $\beta'_{2,3} \neq 0$ and we
consider the product
$$[[x_2, y_1], y_1] = [ \alpha'_{2,2}(B'_1y_2 +
B'_2y_{m_1+1}) +\alpha'_{2,m_1+1}(B_1y_2 + B_2y_{m_1+1})+
\sum\limits_{\begin{array}{c}j=3\\j\neq m_1+1 \end{array}}^m
\alpha_{2,j}y_j, y_1] =$$ $$= \alpha'_{2,2}[B'_1y_2 +
B'_2y_{m_1+1},y_1] + \alpha'_{2,m_1+1}[B_1y_2 + B_2y_{m_1+1},y_1]+
$$ $$+\sum\limits_{\begin{array}{c}j=3\\j\neq m_1+1 \end{array}}^m \alpha_{2,j}[y_j, y_1] =
\alpha'_{2,2} \beta'_{2,3}x_3 + \sum\limits_{i\geq 4}(*)x_4 .$$

 On the other hand, due to (4) we have
 $$[[x_2, y_1], y_1] = \frac 1 2 [x_2, [ y_1, y_1]] =
\frac 1 2 [x_2, \sum\limits_{i=2}^n\beta_{1,i}x_i] =
\sum\limits_{i\geq 4}(*)x_i.$$

Comparing the coefficients at the corresponding basic elements we
get equality $\alpha'_{2,2}\beta'_{2,3} = 0,$ i.e. we have a
contradiction with supposition $h=3.$

If $h \geq 4,$ then we obtain $\beta'_{h-1,3} \neq 0.$ Consider
the chain of equalities
$$[y_{h-2}, y_2] = [y_{h-2}, [y_1, x_1]] = [[y_{h-2}, y_1], x_1] -
[[y_{h-2}, x_1], y_1] =$$ $$= [
\sum\limits_{i=3}^n\beta_{h-2,i}x_i , x_1] - [y_{h-1}, y_1] = -
\beta_{h-1,3}x_3 + \sum\limits_{i\geq 4}(*)x_i.$$

Since $y_{h-2} \in L^{h-1}$ and $y_2 \in L^3$ then $x_3 \in
L^{h+2} = \{x_4, \dots, x_n, y_{h-1}, \dots, y_m\},$ which is a
contradiction with the assumption that the nilindex of $L$ is
equal to $n+m.$
\end{proof}

\begin{rem} In this subsection we used product $[y_1,x_1]=y_2.$ However, it is not difficult
to check that the obtained results are also true under the
condition $[y_1,x_1]=0.$
\end{rem}

\subsection{The case of both generators lie in $L_1$}

\begin{thm}\label{t6}
Let $L=L_0 \oplus L_1$  be a Leibniz superalgebra from
$Leib_{n,m}$ with characteristic sequence equal to $(n_1, \dots,
n_k | m_1,\dots, m_s),$ where $n_1\leq n-2, \ m_1\leq m-1$ and let
both generators lie in $L_1.$ Then $L$ has a nilindex less than
$n+m.$
\end{thm}
\begin{proof}
 Since
both generators of the superalgebra $L$ lie in $L_1,$ they are
linear combinations of the elements $\{y_1, y_{m_1+1}, \dots,
y_{m_1+\dots+m_{s-1}+1}\}.$ Without loss of generality we may
assume that $y_1$ and $y_{m_1+1}$ are generators.

Let $L^{2t} = \{x_i, x_{i+1}, \dots, x_n, y_j, \dots, y_m\}$ for
some natural number $t$ and let $z \in L$ be an arbitrary element
such that $z \in L^{2t} \setminus L^{2t+1}.$ Then $z$ is obtained
 by the products of even number of generators. Hence $z \in L_0$
and $L^{2t+1} = \{x_{i+1}, \dots, x_n, y_j, \dots, y_m\}.$ In a
similar way, having $L^{2t+1} = \{x_{i+1}, \dots, x_n, y_j, \dots,
y_m\}$ we obtain $L^{2t+2} = \{x_{i+1}, \dots, x_n, y_{j+1},
\dots, y_m\}.$

From the above arguments we conclude that $n = m-1$ or $n = m-2$
and
$$L^3 = \{x_2, \dots, x_n, y_2, y_3, \dots, y_{m_1},
y_{m_1+2}, \dots, y_m\}.$$ Applying the above arguments we get
that an element of form $B_1y_2 + B_2y_{m_1+2} + B_3y_{m_1+m_2+1}$
disappears in $L^4.$ Moreover, there exist two elements $B'_1y_2 +
B'_2y_{m_1+2} + B'_3y_{m_1+m_2+1}$ and $B''_1y_2 + B''_2y_{m_1+2}
+ B''_3y_{m_1+m_2+1}$ which belong to $L^4,$ where
$$rank \left(\begin{array}{lll}
B_1&B_2&B_3\\
B'_1&B'_2&B'_3\\
B''_1&B''_2&B''_3\end{array}\right) =3.$$ Since $x_2$ does not
belong to $L^5$ then the elements $B'_1y_2 + B'_2y_{m_1+2} +
B'_3y_{m_1+m_2+1},$ $B''_1y_2 + B''_2y_{m_1+2} +
B''_3y_{m_1+m_2+1}$ lie in $L^5.$ Hence,  from the notations
$$[x_1, y_1] = \alpha_{1,2}(B_1y_2 + B_2y_{m_1+2} +  B_3y_{m_1+m_2+1})
+ \alpha_{1,m_1+2}(B'_1y_2 + B'_2y_{m_1+2} +  B'_3y_{m_1+m_2+1})
+$$$$+ \alpha_{1,m_1+m_2+1}(B''_1y_2 + B''_2y_{m_1+2} +
B''_3y_{m_1+m_2+1})+ \sum\limits_{j=3, j\neq m_1+2, m_1+m_2+1}^m
\alpha_{1,j}y_j.$$
$$[x_1, y_{m_1+1}] = \delta_{1,2}(B_1y_2 + B_2y_{m_1+2} +  B_3y_{m_1+m_2+1})
+ \delta_{1,m_1+2}(B'_1y_2 + B'_2y_{m_1+2} +  B'_3y_{m_1+m_2+1})
+$$$$+ \delta_{1,m_1+m_2+1}(B''_1y_2 + B''_2y_{m_1+2} +
B''_3y_{m_1+m_2+1})+ \sum\limits_{j=3, j\neq m_1+2, m_1+m_2+1}^m
\delta_{1,j}y_j,$$ we have $(\alpha_{1,2},\delta_{1,2}) \neq
(0,0).$

Similarly, from the notations
$$[B_1y_2 + B_2y_{m_1+2} +  B_3y_{m_1+m_2+1}, y_1 ] = \beta_{2,2}x_2 + \beta_{2,3}x_3 + \dots + \beta_{2,n}x_n,$$
$$[B_1y_2 + B_2y_{m_1+2} +  B_3y_{m_1+m_2+1}, y_{m_1+1}] = \gamma_{2,2}x_2 + \gamma_{2,3}x_3 + \dots +
\gamma_{2,n}x_n,$$ we obtain the condition $(\beta_{2,2},
\gamma_{2,2}) \neq (0, 0).$

Consider the product
$$[x_1, [y_1, y_1]] = 2 [[x_1, y_1], y_1] = 2\alpha_{1,2}[B_1y_2 + B_2y_{m_1+2} +  B_3y_{m_1+m_2+1},
y_1]+$$ $$ +2\alpha_{1,m_1+2}[B'_1y_2 +
B'_2y_{m_1+2}+B'_3y_{m_1+m_2+1}, y_1] +$$$$
2\alpha_{1,m_1+m_2+1}[B''_1y_2 + B''_2y_{m_1+2} +
B''_3y_{m_1+m_2+1},y_1]+$$
$$+2\sum\limits_{j=3, j\neq m_1+2, m_1+m_2+1}^m
\delta_{1,j}[y_j,y_1]=2 \alpha_{1,2}\beta_{2,2}x_2 +
\sum\limits_{i\geq 3}(*)x_i .$$ On the other hand,
$$[x_1, [y_1, y_1]] = [x_1, \beta_{1,1}x_1 +
\beta_{1,2}x_2 + \dots + \beta_{1,n}x_n] = \sum\limits_{i\geq
3}(*)x_i.$$

Comparing the coefficients at the basic elements in these
equations we obtain $\alpha_{1,2}\beta_{2,2} = 0.$

Analogously, considering the product $[x_1, [y_{m_1+1},
y_{m_1+1}]],$ we obtain $\delta_{1,2}\gamma_{2,2} = 0.$

From this equations and the conditions $(\beta_{2,2},
\gamma_{2,2}) \neq (0, 0),$ $(\alpha_{1,2},\delta_{1,2}) \neq
(0,0)$ we easily obtain that the solutions are
$\alpha_{1,2}\gamma_{2,2} \neq 0, \beta_{2,2} =\delta_{1,2} =0$ or
$\beta_{2,2}\delta_{1,2}\neq 0, \alpha_{1,2} = \gamma_{2,2}=0.$

Consider the following product
$$[[x_1, y_1], y_{m_1+1}] = [x_1,[ y_1, y_{m_1+1}]] - [[x_1, y_{m_1+1}], y_1]
=-\delta_{1,2}\beta_{2,2}x_2+ \sum\limits_{i\geq 3}(*)x_i.$$ On
the other hand, $$[[x_1, y_1], y_{m_1+1}] =
\alpha_{1,2}\gamma_{2,2}x_2+ \sum\limits_{i\geq 3}(*)x_i.$$
Comparing the coefficients of the basic elements in these
equations we obtain irregular equation $\alpha_{1,2}\gamma_{2,2} =
-\beta_{2,2}\delta_{1,2}.$ It is a contradiction with supposing
the nilindex of the superalgebra equal the $n+m.$ And the theorem
is proved. \end{proof}

Thus, the results of the Theorems \ref{t2}--\ref{t6} show that the
Leibniz superalgebras with nilindex $n+m$ ($m\neq 0$) are the
superalgebras mentioned in section 2. Hence, we completed the
classification of the Leibniz superalgebras with nilindex $n+m.$

\end{document}